\documentclass[12pt]{article}
\usepackage{graphicx,color}

\usepackage{amsmath, amsthm, amssymb}

\newtheorem{Theorem}{Theorem}[section]
\newtheorem{Definition}{Definition}[section]
\newtheorem{Lemma}{Lemma}[section]
\newtheorem{Proposition}[Theorem]{Proposition}

\begin{document}
\title{Structurally Stable Homoclinic Classes}
\author{Xiao Wen}

\maketitle

\begin{abstract}
In this paper we study structurally stable homoclinic classes. In a
natural way, the structural stability for an individual homoclinic
class is defined through the continuation of periodic points. Since
the homoclinic classes is not innately locally maximal, it is hard
to answer whether structurally stable homoclinic classes are
hyperbolic. In this article, we make some progress on this question.
We prove that if a homoclinic class is structurally stable, then it
admits a dominated splitting. Moreover we prove that codimension one
structurally stable classes are hyperbolic. Also, if the
diffeomorphism is far away from homoclinic tangencies, then
structurally stable homoclinic classes are hyperbolic.
\end{abstract}


\section{Introduction}
 An  important notion in dynamical systems  coming from
 Physics and Mechanics is the so called structural stability.
 Precisely, a diffeomorphism $f$ is  {\it structurally stable} if there
is a $C^1$ neighborhood ${\mathcal U}$ of $f$ in ${\rm Diff}(M)$
such that, for every $g\in {\mathcal U}$, there is a homeomorphism
$h: M\to M$ such that $h\circ f=g\circ h$. Since such a
homeomorphism $h$ preserves orbits, a structurally stable system is
one that has robust dynamics, that is, one whose orbital structure
remains unchanged under perturbations.

The non-recurrent part of dynamical systems is fairly robust with
respect to perturbations. But the recurrent part is fragile and, to
survive from perturbations, it needs  the condition of (various
versions of) hyperbolicity. For instance, a single periodic orbit is
structurally stable if and only if it is hyperbolic, meaning no
eigenvalue of modulus 1. For the whole system $f$ to be structurally
stable, a crucial condition needed is that ${\rm CR}(f)$, the set
that captures all the recurrence, is a hyperbolic set. Recall a
compact invariant set $\Lambda\subset M$ of $f$ is called {\it
hyperbolic} if, for
 each $x\in \Lambda$, the tangent space $T_x M$  splits into
 $ T_x M=E^s(x)\oplus E^u(x)$ such that
$$ Df(E^s(x))=E^s(f(x)), \ Df(E^u(x))=E^u(f(x))$$
and, for some constants $C\ge 1$ and $0<\lambda<1$,
$$ |Df^n(v)|\le C\lambda^n |v|, \ \forall x\in \Lambda, v\in E^s(x), n\ge 0,$$
$$ |Df^{-n}(v)|\le C\lambda^n |v|, \ \forall x\in \Lambda, v\in E^u(x), n\ge 0.$$

Briefly, a hyperbolic set is one at which tangent vectors split into
two directions, contracting and expanding upon iterates,
respectively, with uniform exponential rates. This definition
extends the hyperbolicity condition from a single periodic orbit to
a general compact invariant set. It is closely related to structural
stability. Indeed, the following remarkable result, known as the
stability conjecture of Palis and Smale \cite{PS1}, is fundamental
to dynamical systems:

{\bf Theorem} (Ma\~n\'e \cite{Man2}). \ {\it
 If a diffeomorphism $f$ is structurally stable then
${\rm CR}(f)$ is hyperbolic.}

In this paper we consider a more general version of structural
stability. It is for an individual ``basic piece" of the dynamics,
which is not necessarily manifold $M$ or the whole nonwandering set
$\Omega(f)$. A homoclinic class and a chain recurrent class are two
typical types of ``basic pieces''.

Let $M$ be a compact $C^\infty$ Riemannian manifold without
boundary, and $f:M\to M$ be a diffeomorphism.  Denote ${\rm
Diff}(M)$ the space of diffeomorphisms of $M$ with the
$C^1$-topology.

Let $p$ and $q$ be hyperbolic periodic points. We say that $p$ and
$q$ are {\it homoclinically related} and write $p\sim q$, if
$W^s({\rm Orb}(q))$ (resp. $W^u({\rm Orb}(q))$) and $W^u({\rm
Orb}(p))$ (resp. $W^s({\rm Orb}(p))$) have non-empty transverse
intersections. The set $H(p,f)=\overline{\{q:q\sim p\}}$ is called
the {\it homoclinic class} of $p$. Equivalently, $H(p,f)$ is the
closure of set of the transversely intersection points of $W^s({\rm
Orb}(p))$ and $ W^u({\rm Orb}(p))$

A hyperbolic periodic point has its natural ``continuations".
Precisely, let $p\in M$ be a hyperbolic periodic point of $f$ of
period $k$. Then there exist a compact neighborhood $U$ of ${\rm
Orb}(p)$ in $M$ and a $C^1$-neighborhood $\mathcal{U}(f)$ of $f$
such that  for any $g\in \mathcal{U}(f)$, the maximal invariant set
$$\bigcap_{n=-\infty}^{\infty} g^n(U)$$
of $g$ in $U$ consists of a single periodic orbit $O_g$ of $g$ of
the same period as $p$, which is hyperbolic with ${\rm
Ind}(O_g)={\rm Ind}(p)$. Here ${\rm Ind}(p)$ denotes the {\it index}
of $p$, which is the dimension of the stable manifold of $p$. The
neighborhood $U$ can be chosen to be the union of $k$ arbitrarily
small disjoint balls, each containing exactly one point of ${\rm
Orb}(p)$ and one point of $O_g$. This identifies the {\it
continuation} $p_g$ of $p$ under $g$. Thus the notion of
continuation $p_g$ of $p$ is defined for $g$ sufficiently close to
$f$. As usual,  the (unique) homoclinic class of $g$ that contains
$p_g$ is denoted $H(p_g,g)$. Here is a natural general version of
structural stability for a single homoclinic class:

\begin{Definition}
Let $p$ be a hyperbolic  periodic point of $f$.  We say that
$H(p,f)$ is $C^1$-{\rm structurally stable} if there is a
neighborhood $\mathcal{U}$ of $f$ in ${\rm Diff}(M)$ such that, for
every $g\in \mathcal{U}$, there is a homeomorphism $h: H(p,f)\to
H(p_g,g)$ such that $h\circ f|_{H(p,f)}=g\circ h|_{H(p,f)}$, where
$p_g$ is the continuation of $p$.
\end{Definition}

Note that, while $h$ in this definition preserves periodic points of
$H_f(p)$, it is not clear if it preserves individual continuations.
For instance, it is not clear if $h(p)=p_g$. Also, since $h$  is a
homeomorphism only, it does not preserve the hyperbolicity of
periodic points.

In the article, we get some characterization of the structurally
stable homoclinic classes as follows:

\begin{Theorem}\label{Mainthm1}
Let $H(p,f)$ be a structurally stable homoclinic class of $f$, then
there are constants $C>1$ and $0<\lambda<1$ and an integer $m>0$
such that
\begin{enumerate}
\item $H(p,f)$ admit a dominated
splitting $T_{H(p,f)}M=E\oplus F$ with {\rm dim}$E$={\rm ind}$(p)$.
\item For any $q\sim p$,
$$\prod_{i=0}^{k-1}\|Df^m|_{E^s(f^{im}(q))}\|<C\lambda ^k,$$
$$\prod_{i=0}^{k-1}\|Df^{-m}|_{E^u(f^{-im}(q))}\|<C\lambda ^k,$$
where $k=[\pi(q)/m]$ {\rm(}$\pi(q)$ represents the minimal period of
$q$ and $[\cdot]$ represents the integer part{\rm )}.

\end{enumerate}
\end{Theorem}
Here we recall that an invariant set $\Lambda$ of $f$ admits a
dominated splitting if the tangent bundle $T_\Lambda M$ has a
continuous $Df$-invariant splitting $E\oplus F$ and there exists
constants $C>0$, $0<\lambda<1$ such that
$$\|Df^n|_{E(x)}\|/m(Df^n|_{F(x)})\leq C\lambda^n$$
for all $x\in\Lambda$ and $n\geq0$. Where
$$m(A)=\inf\{\|Av\|: \|v\|=1\}$$ denotes the {\it mininorm} of a linear map
$A$. The dominated splitting is a generalization (or a candidate) of
hyperbolic splitting. With some assumptions, we can prove that the
dominated splitting given in Theorem \ref{Mainthm1} is actually
hyperbolic.

\begin{Theorem}\label{Mainthm2}
Let $f$ be a diffeomorphism of $M$ and $p$ be a hyperbolic periodic
point of $f$ of index $1$ or ${\rm dim} M-1$. If the homoclinic
class $H(p,f)$ of $p$ is structurally stable, then $H(p,f)$ is
hyperbolic.
\end{Theorem}

Recall that a diffeomorphism $f$ is called {\it far away from
homoclinic tangency} if there is a neighborhood $\mathcal{U}$ of $f$
such that for any $g\in\mathcal{U}$ and any hyperbolic periodic
points $q$ of $g$, there is no non-transversely intersection of
$W^s(q)$ and $W^u(q)$.

\begin{Theorem}\label{Mainthm3}
Let $f$ be a diffeomorphism of $M$ which is far away from homoclinic
tangency and $p$ be a hyperbolic periodic point of $f$. If the
homoclinic class $H(p,f)$ of $p$ is structurally stable, then
$H(p,f)$ is hyperbolic.
\end{Theorem}

\section{Periodic points in $H(p,f)$}
In this section we give the proof of Theorem \ref{Mainthm1}. To
prove Theorem \ref{Mainthm1} we need consider a kind of typical
diffeomorphisms with some generic properties.

\begin{Proposition}\label{Generic} There is a residual
subset $\mathcal{R}_0\subset {\rm Diff}(M)$ such that every
$f\in\mathcal{R}_0$ satisfies the following conditions:
\begin{itemize}
\item [1.] f is Kupka-Smale, meaning periodic points of $f$ are each hyperbolic
and their stable and unstable manifolds meet transversally $($see
\cite{PM}).

\item [2.] for any pair of
hyperbolic periodic points $p$ and $q$ of $f$, either $H(p, f) =
H(q, f)$ or $H(p, f)\cap H(q, f) =\emptyset$.(see \cite{BC})

\item [3.] if two hyperbolic periodic points $p$ and $q$ of $f$ are in the same
topologically transitive set and ${\rm Ind}(p)\leq{\rm Ind}(q)$,
then $W^s({\rm Orb}(q), f)\pitchfork W^u({\rm Orb}(p), f)$ $\neq
\emptyset$ (see \cite{GW}).

\item [4.] for every pair of periodic
points $p_f$ and $q_f$ of $f$, there exists a neighborhood
$\mathcal{U}$ of $f$ in $\mathcal{R}_0$ such that either
$H(p_g,g)=H(q_g,g)$ for all $g\in\mathcal{U}$ or $H(p_g,g)\cap
H(q_g,g)=\emptyset$ for all $g\in \mathcal{U}$(see \cite{ABCDW}).
\end{itemize}
\end{Proposition}

In the proof we also need the following version of Frank's Lemma
preserving (un)stable manifolds.

\begin{Proposition}\label{Gourmelon} (\cite{Gou}) \ Let $f$ be a diffeomorphism of $M$.
For any $C^1$ neighborhood ${\mathcal U}$ of $f$, there is
$\epsilon>0$ such that, for any pair of hyperbolic periodic points
$p, q\in M$ of $f$ that are homoclinically related, any neighborhood
$U$ of  ${\rm Orb}(q)$ in $M$ not touching ${\rm Orb}(p)$, and any
continuous path of linear isomorphisms $A_{k, t}: T_{f^k q}M\to
T_{f^{k+1} q}M$ that satisfies the following three assumptions:

$(1)$  $A_{k, 0}=D_{f^kq}f$ for all $0\leq k<\pi(q)$,

$(2)$ $\|A_{k, t}- D_{f^k(q)}f\|<\epsilon$ for all $0\leq k<\pi(q)$
and any $t\in[0, 1]$,

$(3)$ $A_{\pi(q)-1, t}\circ A_{\pi(q)-2, t}\circ\cdots\circ A_{0,
t}$ has no eigenvalue on the unit circle for all $t\in[0, 1]$,

there exist a perturbation $g\in\mathcal{U}$ with the following
three properties:

$(A)$  $g=f$ on  $(M\backslash U)\cup{\rm Orb}(q)$,

$(B)$  $D_{f^kq}g=A_{k,1}$ for all $0\leq k<\pi(q)$,

$(C)$  $p$ and $q$ are homoclinically related with respect to $g$.

\end{Proposition}

Let $q$ be a hyperbolic periodic point of a diffeomorphism $f$. We
say an eigenvalue $\mu$ of $D_qf^{\pi(q)}$ is the {\it weakest
stable eigenvalue } of $q$ if $\log|\mu'|\leq\log|\mu|<0$ for all
eigenvalues $\mu'$ of $D_qf^{\pi(q)}$ with $|\mu'|<1$. Similarly we
define the weakest unstable eigenvalue. We say
$\sqrt[\pi(q)]{|\mu|}$ is the {\it normalized weakest stable
eigenvalue } of $q$. We also say an eigenvalue is simple if it has
multiplicity $1$. Applying the above proposition, we can get the
following lemma.

\begin{Lemma}
Let $f$ be a diffeomorphism of $M$ and $q\sim p$ be hyperbolic
periodic points of $f$. Let $\mu$ be a weakest stable eigenvalue of
$q$ w.r.t $g$. Then there exists a diffeomorphism $g$ arbitrarily
close to $f$ with a hyperbolic periodic point $q'\sim p_g$ such that
the weakest stable eigenvalue $\mu'$ of $q'$ (w.r.t $g$) is real and
$\sqrt[\pi(q')]{|\mu'|}$ is arbitrarily close to
$\sqrt[\pi(q)]{|\mu|}$.
\end{Lemma}

\begin{proof}
Since $q\sim p$, one can find points $x\in W^s(p)\pitchfork W^u(q)$
and $y\in W^u(p)\pitchfork W^s(q)$. It follows that ${\rm
Orb}(p)\cup {\rm Orb}(x)\cup {\rm Orb(q)}\cup {\rm Orb}(y)$ is a
hyperbolic set. Applying the shadowing lemma of hyperbolic set, one
can find a hyperbolic periodic point $q'$ of $f$ such that the
period of $q'$ is arbitrarily large and the normalized weakest
stable eigenvalue of $q'$ is arbitrarily close to
$\sqrt[\pi(q)]{|\mu|}$. Without loss of generality, we assume that
the weakest stable eigenvalue of $q'$ is complex. For any
$\varepsilon>0$, by choosing a suitable $q'$ with period large
enough, using the Lemma 6.6 of \cite{BC}, one can get continuous
paths of linear isomorphisms $A_{k, t}: T_{f^k q'}M\to T_{f^{k+1}
q'}M (k=0,1,\cdots, \pi(q')-1)$ with the following properties:

$(1)$  $A_{k, 0}=D_{f^kq'}f$ for all $0\leq k<\pi(q')$,

$(2)$ $\|A_{k, t}- D_{f^k(q')}f\|<\epsilon$ for all $0\leq
k<\pi(q')$ and any $t\in[0, 1]$,

$(3)$ $A_{\pi(q')-1, t}\circ A_{\pi(q')-2, t}\circ\cdots\circ A_{0,
t}$ has no eigenvalue on the unit circle for all $t\in[0, 1]$,

$(4)$ The weakest stable eigenvalue of $A_{\pi(q')-1, 1}\circ
A_{\pi(q')-2, 1}\circ\cdots\circ A_{0, 1}$ are real and simple and
the normalization of the eigenvalue equal to the normalized weakest
stable eigenvalue of $q'$.

Then one can apply the Proposition \ref{Gourmelon} to get the
perturbation $g$ in the lemma.
\end{proof}

Let $p$ be a hyperbolic periodic point of $f$. If $q\in H(p,f)$ has
a simple
 real stable weakest eigenvalue $\lambda$, then we know that there is a
strong stable manifold $W_{loc}^{ss}(q)$ of codimension one in
$W_{loc}^s(q)$. If we choose $loc$ small enough, then
$W_{loc}^s(q)\backslash W_{loc}^{ss}(q)$ has two components $B_1$
and $B_2$. We say that $q$ is a {\it non-ss boundary point} if
 $W^u(p)$ intersect both $B_1$ and $B_2$. The following lemma can be
 found in \cite{SV}. Here we also give a proof for completeness.

\begin{Lemma} \label{Lemma22}
Let $f\in {\rm Diff}(M)$ with a neighborhood $\mathcal{U}\subset
{\rm Diff}(M)$ and $p$ be a hyperbolic periodic point of $f$. There
exist a constant $\delta>0$ such that for any $q\in H(p,f)$, if $q$
has a simple real weakest stable eigenvalue $\lambda$ with
$(1-\delta)^{\pi(q)}<|\lambda|<1$ and $q$ is a non-ss boundary
point, then for any neighborhood $U$ of ${\rm Orb}(q)$, there exist
a diffemorphism $g\in\mathcal{U}$ such that

$1$. $g=f$ in $M\backslash U$,

$2$. there exist two hyperbolic periodic points $q_1,q_2$ of $g$
which are homoclinic related to $p$ such that their orbits contained
in $U$ and $\pi(q_1)=\pi(q_2)$. Furthermore, the periods of
$q_1,q_2$ are either $\pi(q)$ or $2\pi(q)$.

\end{Lemma}

\begin{proof} After an arbitrarily small perturbation, we can assume
that $f$ is ``locally linear'' near the orbit of $q$ in the sense
that there is $r>0$ such that
$$f|_{B_r(f^iq)}=\exp_{f^{i+1}q}\circ D_{f^iq}g\circ
\exp_{f^iq}^{-1}$$ for any $0\leq i<\pi(q)$. Let $E^c(f^iq)\subset
T_{f^iq}M$ be the eigenspace of $D_{q}g^{\pi(q)}$ associated to the
simple real weakest stable eigenvalue $\lambda$ and $E^{ss}(f^iq)$
be the strong stable eigenspace with $E^c(f^iq)\oplus
E^{ss}(f^iq)=E^s(f^iq)$. For $\eta>0$, denote the ball in
$E^c(f^iq)$ of radius $\eta$ about the origin to be $E^c(f^iq,
\eta)$. Similarly we define $E^s(f^iq, \eta)$ and
$E^{ss}(f^iq,\eta)$. Choose $r_0$ small enough such that $E^s(f^iq,
r_0)\subset W^s(f^iq)$ and $E^{ss}(f^iq, r_0)\subset W^{ss}(f^iq)$.
Let $B_1,B_2$ be the two components of $E^s(q,r_0)\setminus
E^{ss}(q,r_0)$. By the assumption we know that there exist
$y_1,y_2\in W^u(p)$ such that $y_1\in B_1$ and $y_2\in B_2$ which
are the transverse intersection points of $W^u(p)$ and $W^s(q)$.

We construct a perturbation $\tilde{g}$ of $f$. Let
$\alpha(x):[0,+\infty)\to[0,+\infty)$ be a bump function which
satisfies (1) $\alpha|_{[0,1/3]}=1$, (2)
$\alpha|_{[2/3,+\infty)}=0$, (3) $0<\alpha|_{(1/3,2/3)}<1$ and (4)
$0\leq \alpha'(x)<4$ for all $x\in[0,+\infty)$. For a small $\eta>0$
and any $0\leq i<\pi(q)$, let $\xi=\sqrt[\pi(q)]{|\lambda|}$ be the
normalized eigenvalue associated to $\lambda$, define a real
function $\beta:T_{f^iq}M\to \mathbb{R}$ by
$$\beta(v)=\xi^{-1}\alpha(|v|/\eta)+(1-\alpha(|v|/\eta)).$$ Thus
 $\beta(v)=\xi^{-1}$ for $|v|\leq\eta/3$,
$1<\beta(v)<\xi^{-1}$ for $\eta/3<|v|<2\eta/3$, and $\beta(v)=1$ for
$|v|\geq2\eta/3$. We always assume $\eta$ much less than $r_0$.
Define a perturbation $\tilde{g}$ of $f$ in $\bigcup_{0\leq
i<\pi(q)}B(f^iq,\eta)$ to be
$$\tilde{g}(x)=\exp_{f^{i+1}q}(\beta(v)
\cdot D_{f^{i}q}f(v)), \ \ v=\exp_{f^iq}^{-1}(x)$$ for $x\in B(f^iq,
\eta)$ for all $0\leq i<\pi(q)$, and define $\tilde{g}(x)=f(x)$ for
$x \notin \bigcup_{0\leq i<\pi(q)}B(f^iq,\eta)$ . It is easy to
check that $\tilde{g}$ is $C^1$ close to $f$ if $\xi$ is
sufficiently close to 1.

One can easily check that
$\tilde{g}^{\pi(q)}|_{\exp_q(E^c(q,\eta/3))}=Id \text {\ or\ } -Id$.
Let $q_1,q_2$ be the end points of $\exp_q(E^c(q,\eta/3))$. We know
that $q_1,q_2$ are periodic points of $\tilde{g}$ with period equal
to $\pi(q)$ or $2\pi(q)$. If we choose $\eta$ small enough, the
negative orbits of $y_{1,2}$ will be unchanged under the
perturbation $\tilde{g}$ and the positive orbits of $y_{1,2}$ will
goes to the orbits of $q_{1,2}$ respectively. We also know that if
$\eta$ is small then the local unstable manifolds of $q_{1,2}$ (with
respect to $\tilde{g}$) will be $C^1$ close to the local unstable
manifold of $q$ (with respect to $f$ ). Since the original $W^u(q)$
(with respect to $f$) has a transverse intersection with $W^s(p)$,
one have that $W^u(q_{1,2})$ (with respect to $\tilde{g}$) will have
a transverse intersection with $W^s(p)$. After an arbitrarily small
perturbation near $q_{1,2}$, we can make $q_{1,2}$ hyperbolic and
$y_{1,2}$ still in the intersection of $W^u(p)$ and $W^s(q_{1,2})$
respectively. This shows the existence of $g$ and ends the proof of
lemma.
\end{proof}

Now we prove that, in a structurally stable homoclinic class,
eigenvalues of periodic orbits are uniformly and robustly away from
the unit circle.

\begin{Proposition}\label{Proposition21} Let $f$ be a diffeomorphism and
let $p\in M$ be a hyperbolic periodic point of $f$. If $H(p,f)$ is
structurally stable, then there are a constant $0<\alpha<1$ and a
neighborhood $\mathcal{U}$ of $f$ such that, for any
$g\in\mathcal{U}$ and any periodic point $q$ of $g$ that is
homoclinically related to $p_g$, the derivative $D_qg^{\pi(q)}$ has
no eigenvalue with modulus in $(\alpha^{\pi(q)},\alpha^{-\pi(q)})$.
\end{Proposition}

\begin{proof}
We prove  by contradiction. Suppose  there is a diffeomorphism $g$
arbitrarily $C^1$ close to $f$ and a periodic point $q_0\in
H(p_g,g)$ homoclinically related to $p_g$ such that
$D_{q_0}g^{\pi(q_0)}$ has an eigenvalue $\mu$ with
$\sqrt[\pi(q_0)]{|\mu|}$ arbitrarily close to $1$.

Without loss of generality we assume that $g\in\mathcal{R}_0$ and
the eigenvalue $\mu$ of $q_0$ is the weakest stable eigenvalue. By
using Lemma 2.1, we also can assume that $\mu$ is real and has
multiplicity $1$.

Let $W_{loc}^{ss}(q_0)$ be the codimension one strong stable
manifold of $q_0$ with respect to $g$. Choose $loc$ small enough
such that $W_{loc}^s(q_0)\backslash W_{loc}^{ss}(q_0)$ has two
components $B_1$ and $B_2$. We can assume that $q_0$ is a non-ss
boundary point. Otherwise, if $q_0$ is not a non-ss boundary point,
without loss of generality, we can assume that $W^u(p)$ transversely
intersect $B_1$ at the point $z$. By the $\lambda$-Lemma we can find
a point $s\in W^s(q_0)\pitchfork W^u(q_0$) which is arbitrarily
close to $z$. It is well known that ${\rm Orb}(q_0)\cup {\rm
Orb}(s)$ is a hyperbolic invariant set of $g$. Replacing $g$ by an
arbitrarily small perturbation we can assume that ${\rm
Orb}(q_0)\cup {\rm Orb}(s)$ also admits a partially hyperbolic
splitting $E^{ss}\oplus E^{cs}\oplus E^u$ where $E^{cs}(q_0)$ is the
eigenspace of $D_{q_0}g^{\pi(q_0)}$ associated to the eigenvalue
$\mu$(See Lemma 4.13 of \cite{BDP} for the construction of the
perturbation). By taking $k>0$ large one can assume that
$W_{loc}^{ss}(g^{k\pi(q_0)}s)$ is closer to $W_{loc}^{ss}(q_0)$ than
$W_{loc}^{ss}(s)$. By the shadowing property of ${\rm Orb}(q_0)\cup
{\rm Orb}(s)$, one can find a hyperbolic periodic $q_0'\sim q_0$
arbitrarily close to $f^{k\pi(q_0)}s$ such that its orbit stays
mostly nearby the orbit of $q_0$. Since the orbit of $q_0'$ is
chosen close to the partially hyperbolic set ${\rm Orb}(q_0)\cup
{\rm Orb}(s)$, the strong stable manifold $W_{loc}^{ss}(q_0')$ is
well defined and so $W_{loc}^s(q_0')\backslash W_{loc}^{ss}(q_0')$
has two component $B_1(q_0)$ and $B_2(q_0)$. We can choose $q_0'$
close to $f^{k\pi(q_0)}s$ enough such that $W^u(q_0)$ crosses
$B_1(q_0)$ and $W^u(s)$ crosses $B_2(q_0)$. By the $\lambda$-Lemma
we know that $q_0'$ is a non-ss boundary point. If the orbit of
$q_0'$ spends enough time nearby the orbit of $q_0$ we can get that
$q_0'$ has a real simple normalized weakest eigenvalue close to
$\sqrt[\pi(q_0)]{|\mu|}$. In this case we replace $q_0$ by $q_0'$.

Now we go back to the proof of proposition. We can separate the
hyperbolic periodic points in $H(p_g,g)$ into several sets. Let $I$
be the set of integers $i$ such that there exist hyperbolic periodic
points of index $i$ in $H(p_g,g)$. Assume
$I=\{i_1,i_2,\cdots,i_s\}$. For any $i_k\in I$, denote $$P^k=\{q\in
H(p_g,g)| q \text{ is a periodic point of index } i_k\}.$$ Choose
$p_k\in P^k$ such that $p_k$ has the smallest period in $P^k$. By
item 4 of Proposition \ref{Generic}, we can find a neighborhood
$\mathcal{V}\subset\mathcal{R}_0$ of $g$ such that for any
$g'\in\mathcal{V}$, $H(p_{g'},g')=H({p_i}_{g'},g')$ for all
$i=1,2,\cdots,s$. Note here that $p_{g'}$ is the continuation of $p$
and ${p_i}_{g'}$ is the continuation of $p_i$ associated to $g'$.

Without loss of generality, we assume that $p\in P^1$ and hence
$q_0$ is also contained in $P^1$. Let
$$N=\max\{\pi(p_1),\pi(p_2),\cdots,\pi(p_s),2\pi(q_0)\}$$ and
$P_N(g)=\{q\in H(p_g,g)| q \text{ is a periodic point with }
\pi(q)\leq N \}$ and $P_N^k=P_N(g)\cap P^k$. It is easy to see that
$P_N(g)$ and $P_N^k$ are finite set. Let
$P_N^1=\{p_1^1,\cdots,p_{l_1}^1\}\cup {\rm Orb}(q_0)$ and
$P_N^k=\{p_1^k,\cdots,p_{l_k}^k\}$ for $k=2,\cdots,s$. By item 3 of
Proposition \ref{Generic} one can get that $p_i^k\sim p_j^k$ for all
$1\leq i<j\leq l_k$ and $k=1,2,\cdots,s$. Pick $x_{kij}\in W^s({\rm
Orb} (p_i^k))\cap W^u({\rm Orb} (p_j^k))$ and $y_{kij}\in W^u({\rm
Orb} (p_i^k))\cap W^s({\rm Orb} (p_j^k))$ for all $1\leq k\leq s$
and $1\leq i<j\leq l_k$. Let $\Gamma$ be the union of all orbits of
$p_i^k$ and all orbits of $x_{kij}$ and all orbits of $y_{kij}$. We
know that $\Gamma$ is a close set and $\Gamma\cap {\rm
Orb}(q_0)=\emptyset$. Now we take a neighborhood $U$ of ${\rm
Orb}(q_0)$ such that $\Gamma\cap U=\emptyset$ and perform a
perturbation $g'$ by Lemma \ref{Lemma22}. With another arbitrarily
small perturbation we can assume that $g'\in\mathcal{V}$. Let
$P_N(g')=\{q\in H(p_{g'},g')| q \text{ is a periodic point of } g'
\text{ with } \pi(q)\leq N \}$. Since $g'\in\mathcal{V}$, we know
that ${p_i}_{g'}\in H(p_{g'},g')$ and all ${p_i^k}_{g'}\in
H(p_{g'},g')$. The periodic point $q_0$ of $g$ disappears in
$H(p_{g'},g')$ but two more periodic points $q_1,q_2$ appear in
$H(p_{g'},g')$. One can easily get that $\sharp P_N(g')\geq \sharp
P_N(g)+1$. This contradict that $H(p_g,g)$ is conjugate to
$H(p_{g'},g')$. This ends the proof of Proposition
\ref{Proposition21}.
\end{proof}

Let $\pi:E\to \Lambda$ be a finite dimensional vector bundle and
$f:\Lambda\to \Lambda$ be a homeomorphism. A continuous map $A:E\to
E$ is called a {\it linear co-cycle} (or {\it bundle isomorphism})
if $\pi A=f\pi$, and if $A$ restricted to every fiber is a linear
isomorphism.  The topology of $\Lambda$ is not relevant to our aim
here, and we assume that $\Lambda$ has the discrete topology. We say
$A$ is {\it bounded} if there is $N>0$ such that ${\rm
max}\{\|A(x)\|, \|A^{-1}(x)\|\}\le N$ for every $x\in \Lambda$,
where $A(x)$ denotes $A|_{E(x)}$. For two linear co-cycles $A$ and
$B$ over the same base map $f:\Lambda\to \Lambda$, define
$$d(A, B)={\rm sup}_{x\in \Lambda}\{\|A(x)-B(x)\|, \|A^{-1}(x)-B^{-1}(x)\|\}.$$

 A periodic point $p\in \Lambda$ of
$f$ is called {\it hyperbolic} with respect to $A$ if $A^{\pi(p)}$
has no eigenvalues of absolute value 1, where ${\pi(p)}$ is the
period of $p$. As usual, we denote the contracting and expanding
subspaces of $p$ to be $E^s(p)$ and $E^u(p)$. Then
$E(p)=E^s(p)\oplus E^u(p)$. If every point in $\Lambda$ is periodic
of $f$, then $A$ is called a {\it periodic  linear co-cycle}. A
bounded periodic linear co-cycle $A$ is called a {\it star system}
if there is $\epsilon>0$ such that any $B$ with $d(B, A)<\epsilon$
has no non-hyperbolic periodic orbits. (This notion corresponds to
that of diffeomorphisms on the manifold $M$ but, since perturbations
on manifolds are less restrictive, the star condition on manifolds
is stronger. In fact a star condition on a manifold implies Axiom A
and no-cycle.) The next fundamental result of Liao and Ma\~n\'e says
that, if $A$ is a star system, then the individual hyperbolic
splittings $E^s(p)\oplus E^u(p)$ of $p\in \Lambda$, put together,
form a dominated splitting. It also gives some estimates for  rates
on periodic orbits.

\begin{Theorem}\label{PLC} (\cite{Liao1}, \cite{Man})
Let  $A:E\to E$ be a bounded periodic linear co-cycle over
$f:\Lambda\to \Lambda$. If $A$ is a star system, then there is
$\epsilon>0$ and three constants $m>0$, $C>0$ and $0<\lambda<1$ such
that, for any linear co-cycle $B$ over  $f$ with $d(B, A)<\epsilon$,
and any periodic point $q$ of $B$, the following conditions are
satisfied:

$(1)$  $\|B^m|_{E^s(q)}\|\cdot\|B^{-m}|_{E^u(f^mq)}\|<\lambda$.

$(2)$ Let $k=[\pi(q)/m]$, then
$$\prod_{i=0}^{k-1}\|B^m|_{E^s(f^{im}(q))}\|<C\lambda^k,$$
$$\prod_{i=0}^{k-1}\|B^{-m}|_{E^u(f^{-im}(q))}\|<C\lambda^k.$$
\end{Theorem}

The two inequalities in Item 2 are usually referred to as
``uniformly contracting (expanding) at the periods" for periodic
orbits. We remark that Liao and  Ma\~n\'e did not use the term ``
linear co-cycles''. Liao worked  (for flows) on tangent bundles of
manifolds, and  Ma\~n\'e worked on
 periodic sequences of linear isomorphisms.

Let $\Lambda$ be the union of periodic orbits of $H(p,f)$ which are
homoclinicly related to $p$. The tangent map $Df:T_\Lambda M\to
T_\Lambda M$ acts as a periodic linear co-cycle over $f$. We verify
that it is a star system in the sense of linear co-cycles hence
Theorem \ref{Mainthm1} will follows.
\bigskip

{\noindent\it Proof of Theorem \ref{Mainthm1}}. Suppose for the
contrary there is a linear co-cycle $A:T_\Lambda M\to T_\Lambda M$
arbitrarily close to $Df$ that has a periodic orbit ${\rm Orb}(q)$
of $f$ which is non-hyperbolic with respect to $A$. We join $A$ with
$Df$ by a path $A_t$ with $A_0=Df$ and $A_1=A$. Since $A$ can be
arbitrarily close to $Df$, we may assume $A_t|_{{\rm Orb}(q)}$
satisfies assumption (2) of Proposition \ref{Gourmelon}, for every
$t\in [0, 1]$. Let $s\in (0, 1]$ be the first parameter that makes
$q$ non-hyperbolic, namely, $q$ is non-hyperbolic with respect to
$A_s$, but is hyperbolic with respect to $A_t$, for every $t\in [0,
s)$. Take $s'$ slightly less than $s$ so that one of the eigenvalues
$\mu$ of $A_{\pi(q)-1, s'}\circ A_{\pi(q)-2, s'}\circ\cdots\circ
A_{0, s'}$ (in absolute value) is within $(\alpha^{\pi(q)},
\alpha^{-\pi(q)})$. Then the path $A_t$, $t\in [0, s']$, satisfies
the three assumptions of Proposition \ref{Gourmelon}, hence there is
$g\in {\mathcal U}$ that preserves ${\rm Orb}(q)$ and ${\rm Orb}(p)$
such that $Dg|_{{\rm Orb}(q)}=A_{s'}|_{{\rm Orb}(q)}$ and such that
$p$ and $q$ are homoclinically related with respect to $g$. Such a
weak eigenvalue $\mu$ contradicts Proposition \ref{Proposition21}.
This verifies that $Df:T_\Lambda M\to T_\Lambda M$ is a star
periodic linear co-cycle over $f$. Thus Theorem \ref{Mainthm1}
follows from Theorem \ref{PLC}.

\section{Structurally stable homoclinic class of generic diffeomorphisms}

In this section, we will prove a weak version of Theorem
\ref{Mainthm2} and Theorem \ref{Mainthm3} under the generic
assumptions.
\begin{Proposition} \label{Prop31}
There exists a residual set $\mathcal{R}_1\subset {\rm Diff}(M)$
such that for any $f\in\mathcal{R}_1$ and any hyperbolic periodic
point $p$ of $f$ with $Ind(p)=1$, the structurally stability of
$H(p,f)$ implies the hyperbolicity.
\end{Proposition}

\begin{proof}
The proof of the above proposition is just as the same as the proof
of the Proposition 5.1 of \cite{WW}. For completeness we give a
sketch of the proof here. Assuming $f\in{\rm Diff}(M)$ and
$Ind(p)=1$,  from Theorem \ref{Mainthm1} we know that if the
homoclinic class $H(p,f)$ is structurally stable then the class
admit a dominated splitting $T_{H(p,f)}M=E\oplus F$ with ${\rm dim}
E=1$. It is proved in \cite{WW} and \cite{Cro2} that there is a
residual set $\mathcal{R}_1\subset {\rm Diff}(M)$ such that if
$f\in\mathcal{R}_1$ and the one dimensional bundle $E$ is not
contracting, then for a minimally non-contracting (w.r.t the bundle
$E$) set $\Lambda\subset H(p,f)$, the dominated splitting $T_\Lambda
M=E\oplus F$ will be partially hyperbolic and
$$\lim_{n\to+\infty}\frac{1}{n}\sum_{i=0}^{n-1}\log\|Df|_{E(f^ix)}\|=0$$
for any $x\in \Lambda$. Then by applying Theorem 4.1 of \cite{WW} we
know that for this partially hyperbolic set $\Lambda$ and any
neighborhood $U$ of $\Lambda$, one can find a periodic orbit $O$
contained in $U\cap H(p,f)$ of index $1$. Thus we can take a
sequence of periodic orbits $Q_n\subset H(p,f)$ such that $Q_n\to
\Lambda$ in the Hausdorff metric. By Proposition \ref{Generic}, each
$Q_n$ is homoclinically related to ${\rm Orb}(p)$. By Theorem
\ref{Mainthm1}, there exist $\lambda\in(0,1)$ and a positive integer
$m$ such that
$$\prod_{i=0}^{k-1}\|Df^m|_{E^c(f^{im}(q))}\|<\lambda^k$$
for  $q\in Q_n$, where $k=[\pi(q)/m]$.  Note that since $K$ is
non-trivial and hence $\pi(Q_n)\to \infty$, by slightly enlarging
$\lambda$ if necessary we may assume $C=1$ in the inequality. Take
$\lambda'\in (\lambda, 1)$. By Pliss's Lemma, there are $q_n\in Q_n$
such that
$$\prod_{i=0}^{j-1}\|Df^m|_{E^c(f^{im}(q_n))}\|<(\lambda')^j$$
for all $j\geq 1$ (One can also find the proof in the proof of Lemma
\ref{goodpoint}). Taking a subsequence if necessary we assume
$q_n\to x\in \Lambda$. Then
$$\limsup_{n\to+\infty}\frac{1}{n}\sum_{i=0}^{n-1}\log\|Df^m|_{E^c(f^{im}(x))}\|<{\rm log}\lambda'.$$
This contradicts the above limit equality. Thus $E$ is  contracting.
Then one can apply the main theorem of \cite{BGY} to get the
hyperbolicity of $H(p,f)$.

\end{proof}

In \cite{CSY}, the authors proved that if $H(p,f)$ is a homoclinic
class of a diffeomorphism far away from homoclinic tangency, then
$H(p,f)$ admits a partially hyperbolic splitting. Furthermore, they
proved the following conclusion as stated in Corollary 1.4 of
\cite{CSY}.

\begin{Proposition}
For generic $f\in {\rm Diff}(M)$ which is far away from homoclinic
tangencies, the homoclinic classes $H(p,f)$ of $f$ satisfy:
\begin{itemize}
\item either $H(p,f)$ is hyperbolic,
\item or $H(p,f)$ contains weak periodic orbits related to $p$.
\end{itemize}
\end{Proposition}

From the above proposition and the Theorem \ref{Mainthm1}, one can
easily get the following proposition.

\begin{Proposition}\label{Prop33}
There exists a residual subset $\mathcal{R}_2\subset {\rm Diff}(M)$
such that for any $f\in\mathcal{R}_2$ which is far away from
homoclinic tangencies and any hyperbolic periodic point $p$ of $f$,
if the homoclinic class $H(p,f)$ of $p$ is structurally stable, then
$H(p,f)$ is hyperbolic.
\end{Proposition}

\section{The proof of Theorem \ref{Mainthm2} and Theorem \ref{Mainthm3}}
From Proposition \ref{Prop31} and Proposition \ref{Prop33}, assuming
$H(p,f)$ is structurally stable, we know that if the index of $p$ is
$1$ or ${\rm dim} M-1$ or the diffeomorphism $f$ is far way from
homoclinic tangencies, then for any neighborhood $\mathcal{U}\subset
{\rm Diff}(M)$ of $f$, there exists $g\in\mathcal{U}$ such that
$H(p_g,g)$ is hyperbolic.

\begin{Lemma}
Let $f\in {\rm Diff}(M)$ and $H(p,f)$ be a structurally stable
homoclinic class of $f$. If for any neighborhood $\mathcal{U}\subset
{\rm Diff}(M)$ of $f$, there exist $g\in\mathcal{U}$ such that
$H(p_g,g)$ is hyperbolic, then $H(p,f)$ satisfies the shadowing
property. Actually, any periodic pseudo orbit in $H(p,f)$ is
shadowed by periodic orbit contained in $H(p,f)$.

\end{Lemma}

\begin{proof}
Let $h:H(p,f)\to H(p_g,g)$ be the conjugate homeomorphism. For any
$\varepsilon>0$ there exists $\varepsilon'>0$ such that for any
$x_1,x_2\in H(p_g,g)$, if $d(x_1,x_2)<\varepsilon'$, then
$d(h^{-1}(x_1),h^{-1}(x_2))<\varepsilon$. Since $H(p_g,g)$ is
hyperbolic, it is locally maximal and admits the shadowing property.
Hence for $\varepsilon'>0$, there exist $\delta'>0$ such that any
$\delta'$-pseudo orbit of $g|_{H(p_g,g)}$ can be $\varepsilon'$
shadowed by a real orbit of $g$ in $H(p_g,g)$. There exists
$\delta>0$ such that for any $y_1,y_2\in H(p,f)$, if
$d(y_1,y_2)<\delta$, then $d(h(y_1), h(y_2))<\delta'$.

Now we let $\{y_i\}_{i=a}^{b}$($-\infty\leq a<b\leq +\infty$) be a
$\delta$-pseudo orbit of $f|_{H(p,f)}$. By the choice of $\delta$ we
know that $\{h(y_i)\}$ should be a $\delta'$-pseudo orbit of
$g|_{H(p_g,g)}$. Then we know that there exists $x\in H(p_g,g)$ such
that $d(g^i(x),h(y_i))<\varepsilon'$ for all $a\leq i\leq b$. By the
choice of $\varepsilon'$, we know that $d(f^i(h^{-1}x);
y_i)<\varepsilon$ for all $a\leq i\leq b$. From the hyperbolicity of
$H(p_g,g)$ we also know that if the $\delta$-pseudo orbit $\{y_i\}$
of $f|_{H(p,f)}$ is periodic then the shadowing points $h^{-1}(x)$
is also periodic.

\end{proof}

From the above lemma one can see that Theorem \ref{Mainthm2} and
Theorem \ref{Mainthm3} is a direct corollary of the following
proposition.

\begin{Proposition}\label{Prop41}
Let $p$ be a hyperbolic periodic point, and $H(p,f)$ be the
homoclinic class of $f$ containing $p$. Assume there exist constants
$N\in\mathbb{N}, 0<\lambda<1$ such that $H(p,f)$ satisfies the
following properties $(P1)$ to $(P3)$.

$(P1)$. There exists a continuous $Df$-invariant splitting
$T_{H(p,f)} M=E\oplus F$ with ${\rm dim}E={\rm ind}(p)$ such that
for every $x\in \Lambda$,
$${\|Df|_{E(x)}\|}/{m(Df|_{F(x)})}<\lambda^2.$$

$(P2)$. For any hyperbolic periodic point $q$, if $q$ is homoclinic
related to $p$ with period $\pi(q)>N$, then
$$\prod_{i=0}^{\pi(q)-1}\|Df|_{E^s(f^i(q))}\|<\lambda ^{\pi(q)}$$
$$\prod_{i=0}^{\pi(q)-1}\|Df^{-1}|_{E^u(f^{-i}(q))}\|<\lambda ^{\pi(q)}.$$

$(P3)$. $f|_\Lambda$ has the shadowing property and every periodic
pseudo orbit can be shadowed by a periodic orbit.

Then $H(p,f)$ is hyperbolic for $f$.

\end{Proposition}

To prove the proposition we prepare some lemmas. The following lemma
is a folklore.
\begin{Lemma}\label{goodpoint}
Let $0<\lambda<1$ be given and let $q$ be a hyperbolic periodic
point of $f$ with $\|Df|_{E^s(x)}\|/m(Df|_{E^u(x)})<\lambda^2$ for
any $x\in {\rm Orb}(q)$. If
$$\prod_{i=0}^{\pi(q)-1}\|Df|_{E^s(f^i(q))}\|<\lambda ^{\pi(q)},$$
$$\prod_{i=0}^{\pi(q)-1}\|Df^{-1}|_{E^u(f^{-i}(q))}\|<\lambda ^{\pi(q)},$$
then there exists $q'\in {\rm Orb}(q)$ such that
$$\prod_{i=0}^{k-1}\|Df|_{E^s(f^i(q'))}\|\leq\lambda ^k,$$
$$\prod_{i=0}^{k-1}\|Df^{-1}|_{E^u(f^{-i}(q'))}\|\leq\lambda ^k,$$
for any $k\geq 1$.
\end{Lemma}

\begin{proof}
Firstly we can prove that there exist $n\in\mathbb{N}$ such that
$$\prod_{i=0}^{k-1}\|Df|_{E^s(f^i(f^nq))}\|\leq\lambda ^k$$ for any
$k\geq 1$. Otherwise, there exist $k_j\geq 1$ such that
$$\prod_{i=0}^{k_j-1}\|Df|_{E^s(f^i(f^jq))}\|>\lambda ^{k_j}$$
for any $j\in\mathbb{N}$. Hence we can take a sequence of positive
integers by $K_1=K_0, K_{i+1}=K_i+k_{K_i}$. Then we take $K_l<K_m$
such that $K_m-K_l$ is divisible by $\pi(q)$. We know that
$$\prod_{i=0}^{K_m-K_l-1}\|Df|_{E^s(f^i(f^{K_l}q))}\|>\lambda ^{K_m-K_l}$$
for all $i\in\mathbb{N}$. This contradicts the inequalities in the
assumption.

Let $0\leq n_1<n_2<\cdots$ be all nature numbers satisfying
$$\prod_{i=0}^{k-1}\|Df|_{E^s(f^i(f^{n_i}q))}\|\leq\lambda ^k$$
for all $k\geq 1$. Similarly, let $0\leq m_1<m_2<\cdots$ be all
nature numbers satisfying
$$\prod_{i=0}^{k-1}\|Df^{-1}|_{E^u(f^{-i}(f^{m_i}q))}\|\leq\lambda^k,$$
for all $k\geq 1$. From the above discussion we know that
$\{n_i\},\{m_i\}$ are all nonempty. We just need to prove that
$\{n_i\}\cap\{m_i\}\neq\emptyset$.

If $\{n_i\}\cap\{m_i\}=\emptyset$, then we can find two integers
$m_a<n_b$ such that $(m_a,n_b)$ contains no integers which belongs
to $\{n_i\}$ and $\{m_i\}$.

 {\noindent\bf
Claim.} For every $m_a\leq l<n_b$,
$$\prod_{j=l}^{n_b-1}\|Df|_{E^s(f^j(q))}\|>\lambda^{n_b-l}$$

\begin{proof}
We will prove the claim by induction. It is easy to check that if
$\|Df|_{E^s(f^{n_b-1}(q))}\|\leq\lambda$, then for any $k\geq 1$,
$$\prod_{j=n_b-1}^{n_b-1+k-1}\|Df|_{E^s(f^j(q))}\|\leq\lambda^k.$$
This contradicts $n_b-1\notin\{n_i\}$. Hence
$\|Df|_{E^s(f^{n_b-1}(q))}\|>\lambda$, therefor the claim is true
for $l=n_b-1$.

Assume the claim is true for $l+1,\cdots,n_b-1$. Hence
$$\prod_{j=n}^{n_b-1}\|Df|_{E^s(f^j(q))}\|>\lambda^{n_b-n},$$
for any $n\geq l+1$. If
$$\prod_{j=l}^{n_b-1}\|Df|_{E^s(f^j(q))}\|\leq\lambda^{n_b-k},$$
then
$$\prod_{j=l}^{l+k-1}\|Df|_{E^s(f^j(q))}\|\leq\lambda^k,$$
for any $k>1$, contradicting with $k\notin\{n_i\}$. Hence
$$\prod_{j=l}^{n_b-1}\|Df|_{E^s(f^j(q))}\|>\lambda^{n_b-l}.$$
This finishes the proof of the claim.
\end{proof}

Using the hypothesis that
$\|Df|_{E^s(x)}\|/m(Df|_{E^u(x)})<\lambda^2$ for any $x\in {\rm
Orb}(q)$, we can get
$\|Df^{-1}|_{E^u(fx)}\|<\lambda^2\|Df|_{E^s(x)}\|^{-1}$. It follows
that for any $m_a\leq l<n_b$,
$$\prod_{j=l+1}^{n_b}\|Df^{-1}|_{E^u(f^j(q))}\|<\lambda^{n_b-l}.$$
Hence $n_b\in\{m_i\}$, contradicting $\{n_i\}\cap\{m_i\}=\emptyset$.
Hence $\{m_i\}\cap\{n_j\}\neq\emptyset$. By taking
$m\in\{m_i\}\cap\{n_j\}$ and $q'=f^m(q)$ we finish the proof of the
lemma.
\end{proof}

If a hyperbolic periodic point $q$ satisfies
$$\prod_{i=0}^{k-1}\|Df|_{E^s(f^i(q))}\|\leq\lambda ^k,$$
$$\prod_{i=0}^{k-1}\|Df^{-1}|_{E^u(f^{-i}(q))}\|\leq\lambda ^k$$
for any $k>0$, we call it a ``good'' hyperbolic periodic point. It
is known that the ``good'' hyperbolic points always has uniform
sized stable and unstable manifolds if it is contained in an
invariant set with dominated splitting(See the Corollary 3.3 of
\cite{PS2}).

\begin{Lemma}\label{pseudoorbit}
Assume that $H(p,f)$ satisfies the hypothesis $(P1)$-$(P3)$ and the
bundle $E$ is not contracting. Then for any constants
$\lambda<\lambda_1<\lambda_2<1$ and $\delta>0$, there exist a
$\delta$-pseudo orbit $\{x_i\}_{i=0}^{i=n}$ in $H(p,f)$ such that:

\begin{enumerate}
\item $\prod_{i=0}^{k-1}\|Df|_{E(x_i)}\|<\lambda_2^k$ and $\prod_{i=k}^n\|Df^{-1}|_{F(x_i)}\|<\lambda_2^{n-k+1}$, for $k=1,2,...,n$.
\item $\prod_{i=0}^{n-1}\|Df|_{E(x_i)}\|>\lambda_1^n.$
\item $x_0=x_n$ is a ``good'' hyperbolic periodic point.
\end{enumerate}
\end{Lemma}

\begin{proof}
Since $E$ is not contracting, one can find a ``bad'' point $b$ in
$H(p,f)$ such that
$$\prod_{j=0}^{n-1}\|Df|_{E(f^j b)}\|\geq 1$$
for all $n\geq 1$.

Since $H(p,f)$ is a homolinic class, there is a hyperbolic periodic
point $q$ homoclinicly related with $p$ such that ${\rm Orb}(q)$
forms a $\delta$-net of $\Lambda$, i.e., for every $a\in\Lambda$,
there is $x\in {\rm Orb}(q)$ such that $d(a,x)<\delta$. By the
hypothesis (P2) and lemma \ref{goodpoint}, we can assume that $q$ is
a ``good'' hyperbolic periodic point.

Now we construct the pseudo orbit $\{x_i\}_{i=0}^{i=n}$ by combining
the orbit of $q$ with the orbit of $b$.

Let $x_0=q, x_1=f(q), x_2=f^2(q), \cdots, x_{s\cdot\pi(q)}=q$ where
$s$ will be chosen big enough. Let $0<t_1\leq \pi(q)$ be an integer
such that $d(f^{t_1},b)<\delta$. Then we let
$x_{s\cdot\pi(q)+1}=f(q), x_{s\cdot\pi(q)+2}=f^2(q), \cdots,
x_{s\cdot\pi(q)+t_1-1}=f^{t_1-1}(q), x_{s\cdot\pi(q)+t_1}=b$. By the
choice of $q$ we know that
$$\frac{1}{k}\sum_{i=0}^{k-1}\log(\|Df|_{E(x_i)}\|)<\log\lambda<\frac{\log\lambda_1+\log\lambda_2}{2},$$
for every $1\leq k\leq s\pi(q)+t_1$.

Choose $l$ such that
$$\frac{1}{s\pi(q)+t_1+l}(\sum_{i=0}^{s\pi(q)+t_1-1}\log(\|Df|_{E(x_i)}\|)+\sum_{i=0}^{l-1}\log(\|Df|_{E(f^i(b))}\|))\geq\frac{\log\lambda_1+\log\lambda_2}{2},$$
and for any $l'<l$,
$$\frac{1}{s\pi(q)+t_1+l}(\sum_{i=0}^{s\pi(q)+t_1-1}\log(\|Df|_{E(x_i)}\|)+\sum_{i=0}^{l-1}\log(\|Df|_{E(f^i(b))}\|))<\frac{\log\lambda_1+\log\lambda_2}{2}.$$
The existence of $l$ is ensured by the property of $b$. Then we take
$x_{s\cdot\pi(q)+t_1+1}=f(b),
\cdots,x_{s\cdot\pi(q)+t_1+l-1}=f^{l-1}(b)$.

Take $0\leq t_2<\pi(q)$ such that $d(f^l(b),f^{-t_2}(q))<\delta$ and
$x_{s\pi(q)+t_1+l}=f^{-t_2}(q),x_{s\pi(q)+t_1+l+1}=f^{-t_2+1}(q),\cdots,x_{s\pi(q)+t_1+l+t_2}=q.$
This ends the construction of $\{x_i\}$. Roughly speaking, the
$\delta$-pseudo orbit $\{x_i\}$ is composed by some repeat of ${\rm
Orb}(q)$ and a segment of ${\rm Orb}(b)$.

Denote $S(k)=\sum_{i=0}^{k-1}\log(\|Df|_{E(x_i)}\|)$ and
$n=s\pi(q)+t_1+l+t_2$. Since $\log(\|Df|_{E(x)}\|)$ have a uniform
upper bound and $$S(s\cdot\pi(q)+t_1+l)\geq
(s\cdot\pi(q)+t_1+l)\cdot\frac{\log\lambda_1+\log\lambda_2}{2}.$$
The inequality $S(n)>n\cdot\lambda_1$ will be true by choosing $s$
big enough.

By the construction of $\{x_i\}$,
$$\frac{1}{k}S(k)<\frac{\log\lambda_1+\log\lambda_2}{2}$$ for any
$0<k<s\pi(q)+t_1+l$. Since $\log(\|Df|_E(x)\|)$ has a uniform upper
bound and $t_2<\pi(q)$, one can choose $s$ big enough such that for
any $0<k\leq n$, $\frac{1}{k}S(k)<\log\lambda_2$.

Denote $T(k)=\sum_{i=n-k+1}^n\log(\|Df^{-1}|_{F(x_i)}\|)$. Since $q$
is a ``good'' hyperbolic periodic point,
$$\frac{1}{k}T(k)<\log\lambda(<\log\lambda_2)$$ for any $k=1,2,\cdots,t_2+1$.
By the choice of $l$, we can get that
$$\frac{S(s\pi(q)+t_1+l)-S(k)}{s\pi(q)+t_1+l-k}>\frac{\log\lambda_1+\log\lambda_2}{2}$$
for any $k<s\pi(q)+t_1+l$. The hypothesis (P1) means that
$\log(\|Df^{-1}|_{F(f(x))}\|)<2\log\lambda-\log(\|Df|_{E(x)}\|)$.
Therefore
$$\frac{T(t_2+1+k)-T(t_2+1)}{k}<2\log\lambda-\frac{\log\lambda_1+\log\lambda_2}{2}<\log\lambda$$
for any $0<k<s\pi(q)+t_1+l$. Hence for every $0<k\leq n$
$$\frac{1}{k}T(k)<\log\lambda<\log\lambda_2.$$ This ends the proof
of the lemma.
\end{proof}

Now we prove Proposition \ref{Prop41}. Suppose that $H(p,f)$
satisfies $(P1)$-$(P3)$ and the bundle $E$ is not contracting. Fix
two constants $\lambda<\lambda_1<\lambda_2<1$. From Lemma
\ref{pseudoorbit} we know that for any $\delta>0$, there exists a
$\delta$-pseudo orbit $\{x_i\}_{i=0}^{n}$ with the properties in
Lemma \ref{pseudoorbit}. Let $q=x_0=x_n$ be the ``good'' periodic
point. By the assumption $(P3)$ we know that there exists a periodic
point $q'$ close to $q$ such that the orbit of $q'$ shadows the
pseudo orbit $\{x_i\}_{i=0}^n$.

If we choose $\delta$ small enough, from the properties (1) and (2)
in Lemma \ref{pseudoorbit} we know that $q'$ is also a ``good''
periodic point with respect to $\lambda_2$. It follows that $q'$ and
$q$ all have uniform sized stable and unstable manifolds, therefor
$q'$ will be homoclinic related with $q$ if $q'$ is close to $q$
enough. Hence we can get a $q'\in H(p,f)$ with large period (at
least $l$) such that
$\prod_{i=0}^{\pi(q')-1}\|Df|_{E^s(f^i(q'))}\|>\lambda_1^{\pi(q')}$,
contradicting hypothesis (P2). This proves Proposition \ref{Prop41}.

\end{document}